\documentclass[12pt]{article}

\usepackage{amsthm}
\usepackage{amssymb}
\usepackage{amsmath}
\usepackage{amsfonts}
\usepackage{times}

\usepackage[paper=a4paper, top=2.5cm, bottom=2.5cm, left=2.5cm, right=2.5cm]{geometry}

\def\qed{\hfill $\vcenter{\hrule height .3mm
\hbox {\vrule width .3mm height 2.1mm \kern 2mm \vrule width .3mm
height 2.1mm} \hrule height .3mm}$ \bigskip}

\def \RR {\mathbb R}

\def \EE {\mathbb E}
\def \ZZ {\mathbb Z}
\def \Var {\mathrm{Var}}
\def \WW {\mathcal{M}}

\def \eps {\varepsilon}

\def \V {\mathrm{V}}

\def \DC {\mathcal{C}_n}

\def \Id {\mathrm{I}_n}
\def \COV {\mathrm{Cov}}
\def \SIGN {\mathrm{sign}^n}

\def \Cor {\mathrm{Cor}}

\def \ff {{\tilde f}}
\def \gg {{\tilde g}}
\def \FF {\mathcal{F}}

\def \I {\mathrm{I}}
\def \Tr {\mathrm{Tr}}

\def \KL {\mathcal{D}_{KL}}

\newtheorem{theorem}{Theorem}
\newtheorem{lemma}[theorem]{Lemma}

\newtheorem{fact}[theorem]{Fact}

\newtheorem{conjecture}[theorem]{Conjecture}
\newtheorem{proposition}[theorem]{Proposition}

\theoremstyle{definition}

\theoremstyle{remark}
\newtheorem{remark}[theorem]{Remark}

\long\def\symbolfootnotetext[#1]#2{\begingroup
\def\thefootnote{\fnsymbol{footnote}}\footnotetext[#1]{#2}\endgroup}

\title{Second-order bounds on correlations between increasing families}
\author{Ronen Eldan\thanks{Weizmann Institute of Science. Supported by a European Research Council Starting Grant (ERC StG) and by the an Israel Science Foundation grant no. 715/16}}
\begin{document}

\maketitle
\begin{abstract}
Harris's correlation inequality states that any two monotone functions on the Boolean hypercube are positively correlated. Talagrand \cite{Talcorr} started a line of works in search of quantitative versions of this fact by providing a lower bound on the correlation in terms of the influences of the functions. A famous conjecture of Chv\'{a}tal \cite{Chvatal} was found by Friedgut, Kahn, Kalai and Keller \cite{FKKK} to be equivalent to a certain strengthening of Talagrand's bound, conjectured to hold true when one of the functions is antipodal (hence $g(x) = 1-g(-x)$). Motivated by this conjecture, we strengthen some of those bounds by giving estimates that also involve the second order Fourier coefficients of the functions. In particular we show that in the bounds due to Talagrand and due to Keller, Mossel and Sen \cite{KMS14}, a logarithmic factor can be replaced by its square root when one of the functions is antipodal. Our proofs follow a different route than the ones in the literature, and the analysis is carried out in the Gaussian setting.
\end{abstract}

\section{Introduction}
Define $\DC = \{-1,1\}^n$, and denote the uniform measure on $\DC$ by $\mu$. For a function $f:\DC \to \{0,1\}$, which is usually called a Boolean function, we define $\EE_\mu[f] = \int_{\DC} f(x) d \mu(x)$ and often abbreviate $\EE=\EE_\mu$. The discrete derivatives of a function are defined by
$$
\partial_i f(x) = f(x;x_i \to 1) - f(x;x_i \to -1).
$$
and the discrete gradient $\nabla f(x) = \left ( \partial_1 f(x),...,\partial_n f(x) \right )$. We say that a Boolean function is \emph{increasing} if $\partial_i f(x) \geq 0$ for all $i \in [n]$ and all $x \in \DC$.

Harris's correlation inequality states that any two increasing functions $f$ and $g$ must like each other, in the sense that one has $\Cor(f,g) \geq 0$, where
$$
\Cor(f,g) := \EE [fg]-\EE[f]\EE[g].
$$
Talagrand \cite{Talcorr} initiated a line of work which attempts to quantify to what extent this inequality holds true, which is also the topic of the present work. 

The \emph{influence} of the $i$-th coordinate is defined as
$$
\I_i(f) := 2 \int |\partial_i f(x)| d \mu(x).
$$
Talagrand \cite{Talcorr} proved that for increasing $f,g$
\begin{equation}\label{eq:tal1}
\Cor(f,g) \geq  c \frac{ \sum_i \I_i(f) \I_i(g) }{ \log \left ( \frac{e}{ \sum_i \I_i(f) \I_i(g)} \right )},
\end{equation}
where in the last inequality, as in the rest of this paper, the letter $c$ will denote a positive universal constant whose value may change between different appearances.

A rather similarly-looking inequality by Keller, Mossel and Sen \cite{KMS14} states that
\begin{equation}\label{eq:KMS}
\Cor(f,g) \geq c \sum_i \frac{\I_i(f) \I_i(g) }{ \sqrt{\log(e/\I_i(f)) \log (e/\I_i(g)) } } .
\end{equation}
As explained in \cite{KKM15}, both of those bounds are sharp, and none of the two implies the other. In the same paper strengthening are obtained in the case the the functions exhibit some symmetries.

A function $f$ is called \emph{antipodal} if $f(x) = 1-f(-x)$ for all $x$. In \cite{FKKK}, the authors prove that the following conjecture is equivalent to the well-known Chv\'atal's conjecture in combinatorics (for a formulation of the original form of the conjecture, we refer to their paper).  
\begin{conjecture} (\cite{FKKK})
If $f,g$ are increasing and $g$ is antipodal, then
$$
\Cor(f,g) \geq \frac{1}{4} \min_i \I_i(f).
$$
\end{conjecture}

The objective of the present work is twofold. First, we introduce a new approach for obtaining the inequalities \eqref{eq:tal1} and \eqref{eq:KMS}. Second, motivated by the above conjecture, we give a refined bound in terms of the second-degree Fourier coefficients of the functions $f,g$ which allow us to exploit the antipodality of the function. This will allow us to prove the following theorem.

\begin{theorem} \label{thm:symm}
If $f,g$ are increasing and $g$ is antipodal, then
\begin{equation}\label{eq:mainsymm}
\Cor(f,g) \geq c \frac{ \sum_{i \in [n]} \I_i(f) \I_i(g) }{ \sqrt{\log\left ( \frac{2e}{  \sum_{i \in [n]} \I_i(f) \I_i(g) } \right )} }.
\end{equation}
\end{theorem}
Compared to Talagrand's bound \eqref{eq:tal1}, the logarithmic factor in the denominator is replaced by its square root. Theorem \ref{thm:symm} follows from a more general second-order refinement of Talagrand's bound. Define
$$
\V(f)_{i,j} = \EE \left [\partial_i \partial_j f \right ],
$$
which is also the matrix of second-order Fourier coefficients of the function $f$. Our refinement reads,
\begin{theorem} \label{thm:mainboundTal}
For any increasing Boolean functions $f,g$,
\begin{equation}\label{eq:mainboundTal}
\Cor(f,g) \geq c \min \left ( \frac{\WW_1(f,g)} { \sqrt{\log\left ( \frac{e}{\WW_1(f,g)} \right )} } , \frac{\WW_1(f,g)^2}{ |\WW_2(f,g)| }  \right )
\end{equation}
where
$$
\WW_1(f,g) = \sum_i \I_i(f) \I_i(g), ~~ \WW_2(f,g) = \langle \V(f), \V(g) \rangle_{HS}.
$$
\end{theorem}
\begin{remark}
It follows from \cite[Theorem 2.4]{Talcorr} that, in the notation of the theorem, one has 
$$
\WW_2(f,g) \leq C \WW_1(f,g) \log \left ( \frac{e}{\WW_1(f,g)} \right ).
$$
Therefore, equation \eqref{eq:mainboundTal} is a strict strengthening of Talagrand's bound \eqref{eq:tal1}.
\end{remark}

The above theorem implies Theorem \ref{thm:symm} via the following observation.
\begin{fact} \label{fact:symm}
	If $f(x)$ is antipodal then $\V(f) = 0$.
\end{fact}
\begin{proof}
If $f$ is antipodal, its Fourier transform is supported on odd degrees, and hence the second degree coefficients vanish.	
\end{proof}
\vspace{-10pt} \noindent
The combination of the above fact with the bound \eqref{eq:mainboundTal} immediately gives \eqref{eq:mainsymm}. For $i \in [n]$ denote 
$$
\V_i(f) = \left ( \EE \partial_1 \partial_i f, \dots, \EE \partial_n \partial_i f  \right ).
$$
We also prove the following improvement of the bound \eqref{eq:KMS}.
\begin{theorem} \label{thm:mainbound}
For any increasing Boolean functions $f,g$,
\begin{equation}\label{eq:mainbound}
\Cor(f,g) \geq c \sum_{i \in [n]} \I_i(f) \I_i(g) \min \left ( \frac{1}{\sqrt{\log\left ( \frac{e}{\I_i(f)\I_i(g)} \right )}}, ~ \frac{\I_i(f) \I_i(g)}{ |\langle \V_i(f), \V_i(g) \rangle| }  \right ).
\end{equation}
\end{theorem}

\subsection{The Gaussian setting}
Our two main theorems will be consequences of respective counterparts on Gaussian space. Gaussian counterparts of the bounds \eqref{eq:tal1} and \eqref{eq:KMS} were already proven in \cite{KMS14}, where the proofs rely on the Boolean bounds as a black box. Our approach works directly on Gaussian space, which results in (arguably) simpler arguments. 

Denote by $\gamma$ the standard Gaussian measure on $\RR^n$. For a function $f:\RR^n \to \RR$, define
\begin{equation}\label{eq:defm1}
\WW_1(f) := \int x f(x) d \gamma(x), ~~ \WW_2(f) := \int \left (x^{\otimes 2} - \Id \right ) f(x) d \gamma(x)
\end{equation}
the first and second degree Hermite tensors of the function $f$. We say that a function $f:\RR^n \to \RR$ is monotone if for all $i \in [n]$, it is monotone in the $i$-th coordinate when the other $[n] \setminus \{i\}$ coordinates are kept fixed. Analogously to the Boolean setting, we set
$$
\Cor(f,g) = \int f g d \gamma - \left ( \int f d \gamma\right )\left ( \int g d \gamma\right ).
$$
The Gaussian analog of Theorem \ref{thm:mainboundTal} reads,
\begin{theorem} \label{thm:mainboundTalGauss}
	For any increasing functions $f,g:\RR^n \to [0,1]$,
	\begin{equation}\label{eq:mainboundTalG}
	\Cor(f,g) \geq c \min \left ( \frac{\WW_1(f,g)} { \sqrt{\log\left ( \frac{e}{\WW_1(f,g)} \right )} } , \frac{\WW_1(f,g)^2}{ |\WW_2(f,g)| }  \right )
	\end{equation}
	where
	$$
	\WW_1(f,g) = \left \langle \WW_1(f), \WW_1(g) \right \rangle, ~~ \WW_2(f,g) = \langle \WW_2(f), \WW_2(g) \rangle_{HS}.
	$$
\end{theorem}
Moreover, our Gaussian analog of Theorem \ref{thm:mainbound} reads,
\begin{theorem} \label{thm:mainboundG}
For any increasing functions $f,g:\RR^n \to [0,1]$,
\begin{equation}\label{eq:mainboundG}
\Cor(f,g) \geq c \sum_{i \in [n]} \min \left ( \frac{\WW_1(f)_i \WW_1(g)_i} { \sqrt{\log\left ( \frac{e}{\WW_1(f)_i \WW_1(g)_i} \right )} } , \frac{(\WW_1(f)_i \WW_1(g)_i)^2}{ |\left \langle \WW_2(f) e_i, \WW_2(g) e_i \right \rangle| }  \right )
\end{equation}
\end{theorem}

In order to see how those two theorems imply their discrete analogs, define $\SIGN(x) = (\mathrm{sign}(x_1),\dots,\mathrm{sign}(x_n))$. For two Boolean functions $f,g$, define $\ff(x) = 2 f(\SIGN(x)) - 1, \gg(x) = 2 g(\SIGN(x)) - 1$ as functions on $\RR^n$. It is easy to verify that
$$
\Cor(f,g) = \frac{1}{2} \Cor(\tilde f, \tilde g).
$$
Moreover, a straightforward calculation gives that when $f$ is monotone,
$$
\WW_1(\ff)_i = \int_{\RR^n} x_i (2 f(\SIGN(x)) - 1) d \gamma(x) = \sqrt{\frac{2}{\pi}} \int_{\DC} x (2f(x) - 1) d \mu(x) = \sqrt{\frac{2}{\pi}} \I_i(f),
$$
and
$$
\WW_2(\ff) e_i = \int_{\RR^n} x_i x (2 f(\SIGN(x)) - 1) d \gamma(x) = \frac{2}{\pi} \int_{\DC}x_i x (2 f(x) - 1) d \mu(x)= \frac{2}{\pi} \V_i(f).
$$
In light of the last three displays, Theorems \ref{thm:mainboundTal} and \ref{thm:mainbound} follow by applying the respective Gaussian variants on the functions $\tilde f$ and $\tilde g$. \\ \\
\textbf{Acknowledgements.} We thank Gil Kalai, Noam Lifshitz and Nathan Keller for useful suggestions.

\section{Preliminaries and stochastic constructions}
In this section we define several processes which serve as core ingredients of our proofs. Those processes can be thought of as continuous versions of the jump process constructed in a recent paper of Gross and the author \cite{EldanGrossBool}, and some of the ideas are analogous to the ones that appear there. However, unlike the case of \cite{EldanGrossBool} where the pathwise analysis is an essential part of the proof, most of the steps here are carried out in expectation. 

\subsection{Stochastic processes}
Let $B_t$ be a standard Brownian motion in $\RR^n$, adapted to a filtration $\FF_t$. Define
$$
Z_t := \int_0^t e^{-s/2} d B_s.
$$
We have almost surely $[Z]_1 - [Z]_t = \int_t^1 e^{-s} ds = e^{-t}$, concluding that $Z_\infty \sim \gamma$ and that $Z_\infty | Z_t \sim \mathcal{N}(Z_t, e^{-t} \Id)$. Denote by $m_t(x)$ the density of the law of $Z_\infty | Z_t$ with respect to the Lebesgue measure, so that
$$
m_t(x) = e^{n t/2} (2 \pi)^{-n/2} \exp \left ( - \frac{1}{2} e^t \left | x - Z_t\right |^2 \right ).
$$
For a function $h: \RR^n \to \RR$, consider the martingale
$$
M_t = M^{(h)}_t := \EE[h(Z_\infty)|Z_t] = \int h(x) m_t(x) dx.
$$
It\^o's formula gives
$$
d m_t(x) = e^{t/2} \langle  x - Z_t, d B_t \rangle m_t(x) 
$$
so that
\begin{equation}\label{eq:dMt0}
d M_t =  e^{t/2} \int h(x) \langle x-Z_t, d B_t \rangle m_t(x) dx = e^{-t/2} \left \langle M_t^{(h,1)}, d B_t \right \rangle
\end{equation}
where
$$
M_t^{(h,1)} := e^{t} \int h(x) (x - Z_t) m_t(x) dx.
$$
For two functions $f,g: \RR^n \to \RR$, since $Z_\infty \sim \gamma$, we have
$$
\Cor(f,g) = \EE \left [M_\infty^{(f)} M_\infty^{(g)} \right ].
$$
By formula \eqref{eq:dMt0} and by It\^{o}'s isometry, we have
\begin{equation}\label{eq:Cor1}
\Cor(f,g) = \int_0^\infty \EE d [M^{(f)}, M^{(g)}]_t = \int_0^\infty e^{-t} \EE \left [ \left  \langle  M_t^{(f,1)}, M_t^{(g,1)}  \right  \rangle \right ] dt.
\end{equation}
As we will see later on, one has
\begin{fact} \label{fact:mon}
If $f$ is monotone, then one has almost surely for all $t$, $M_t^{(f,1)} \in \RR_+^n$.
\end{fact}
A consequence of the above fact is that $\left  \langle  M_t^{(f,1)}, M_t^{(g,1)}  \right  \rangle \geq 0$ almost surely, for all $t$. Thus, equation \eqref{eq:Cor1} readily implies the following.
\begin{proposition} \label{prop:intcov}
For all $t>0$,
$$
\Cor(f,g) \geq \COV(f(Z_t), g(Z_t)) = \int_0^t e^{-s} \EE \left [ \left  \langle  M_s^{(f,1)}, M_s^{(g,1)}  \right  \rangle \right ] ds. 
$$
\end{proposition}

Our estimates will amount to bounding the right hand side of the above inequality. In order to do so, we need to derive expressions for higher stochastic time-derivatives of the processes $M_t^{(f)}, M_t^{(g)}$. 

\subsection{Higher derivatives}
Let us calculate the higher derivatives of those processes. Similar calculations have been carried out in \cite{Eldan-twosided}, but we include them for the sake of completeness. Define $L_t(x) = e^{-t/2}x + Z_t$ so that the push forward of $\gamma$ under $L_t$ has density $m_t$. Also write 
\begin{equation}\label{eq:defht}
h_t = h \circ L_t
\end{equation}
so that the last display with a change of variables gives
\begin{equation}\label{eq:altdefmt}
M_t^{(h)} = \int h_t(x) d \gamma(x)
\end{equation}
and
\begin{equation}\label{eq:defmt1}
M_t^{(h,1)} = e^{t/2} \int x h_t(x) d \gamma(x).
\end{equation}
\begin{proof} [Proof of fact \ref{fact:mon}]
If the function $h$ is monotone then so is $h_t$. In light of formula \eqref{eq:defmt1} it is clear that the coordinates of $M_t^{(1)}$ are non-negative.
\end{proof}
We now define the processes corresponding to higher cumulants as
\begin{equation}\label{eq:defmtk}
M_t^{(k)} = M_t^{(h,k)} := e^{kt/2} \int h_t(x) H^{(k)}(x) d \gamma(x)
\end{equation}
where $H^{(k)}$ is the $k$-th Hermite tensor, hence $H^{(0)}(x) = 1, H^{(1)}(x) = x$ $H^{(2)}(x) = x^{\otimes 2} - \Id, (H^{(2)}(x))_{i,j,k} = x_i x_j x_k - \delta_{i,j} x_k - \delta_{i,k} x_j - \delta_{j,k} x_i$, etc. Remark that, by definition,
\begin{equation}\label{eq:MtMh}
M_0^{(h,1)} = \WW_1(h), ~~~ M_0^{(h,2)} = \WW_2(h).
\end{equation}
The next lemma gives a formula for derivatives of any order.
\begin{lemma}
We have, almost surely for all $t \geq 0$ and all $k \in \ZZ$,
\begin{equation}\label{eq:dmkt}
d M_t^{(k)} = e^{-t/2} M_t^{(k+1)} d B_t.
\end{equation}	
\end{lemma}

\begin{proof}
First assume that $h$ is smooth enough. Integration by parts gives that
\begin{equation}\label{eq:byparts}
\int \nabla^k h(x) d \gamma(x) = \int h(x) H^{(k)}(x) d \gamma(x).
\end{equation}
This gives that
$$
M_t^{(h,k)} := e^{kt/2} \int \nabla^k (h (L_t (x))) d \gamma(x) = \int (\nabla^k h) (L_t (x)) d \gamma(x).
$$
Plugging the function $\nabla^k h$ in place of $h$ in equation \eqref{eq:dMt0} gives
\begin{align*}
d \int (\nabla^k h) (L_t (x)) d \gamma(x) ~& \stackrel{\eqref{eq:altdefmt}}{=} d M^{(\nabla^k h,0)} \\
& \stackrel{\eqref{eq:dMt0}}{=} M_t^{(\nabla^k h,1)} d B_t \\
& \stackrel{ \eqref{eq:defmt1}}{=} \int x \otimes \bigl ((\nabla^k h) (L_t(x)) \bigr) d \gamma(x) d B_t \\
& = \int \nabla \bigl ((\nabla^k h) (L_t(x)) \bigr) d \gamma(x) d B_t \\
& = e^{-t/2} \int \bigl ((\nabla^{k+1} h) (L_t(x)) \bigr) d \gamma(x) d B_t \stackrel{ \eqref{eq:byparts} }{=} \int h(x) H^{(k)}(x) d \gamma(x) d B_t.
\end{align*}
Equation \eqref{eq:dmkt} follows. In the general case (with no smoothness assumtions), equation \eqref{eq:dmkt} can then be obtained by an approximation argument, but it can also be obtained directly by a straightforward but somewhat tedious calculation using It\^{o}'s formula. For a more rigorous derivation, we refer the reader to \cite{Eldan-twosided}. 
\end{proof}

For two functions $\ff, \gg : \RR^n \to \RR$ and $k \in \ZZ$, define
\begin{equation}\label{eq:defp}
S_t^{(k)} = \left \langle M_t^{(\ff,k)}, M_t^{(\gg,k)} \right \rangle \mbox{ and } p_k(t) := \EE \left [S_t^{(k)} \right ].
\end{equation}
By Ito's formula and \eqref{eq:dmkt}, we have
\begin{align*}
d S_t^{(k)} ~& = M_t^{(\ff,k)} d M_t^{(\gg,k)} + M_t^{(\gg,k)} d M_t^{(\ff,k)} + d[M^\ff, M^\gg]_t \\
& = e^{-t/2} \left (\langle M_t^{(\gg,k)}, M_t^{(k+1,\ff)} d B_t \rangle + \langle M_t^{(\ff,k)}, M_t^{(k+1,\gg)} d B_t \rangle \right ) + e^{-t} \left \langle M_t^{(k+1,\ff)}, M_t^{(k+1,\gg)} \right  \rangle dt.
\end{align*}
By taking expectations, we get
\begin{equation}\label{eq:ptdd1}
p_k'(t) = e^{-t} p_{k+1}(t)
\end{equation}
and by differentiating twice and using the same formula, we finally get
\begin{equation}\label{eq:ptdd2}
p_k''(t) = - p_k'(t) + e^{-2t} p_{k+2}(t).
\end{equation}

\section{Level inequalities}
The main purpose of this section is to prove inequalities which will be used to establish bounds between different time-derivatives of the stochastic processes. As equation \eqref{eq:dmkt} suggests, those will boil down to relations between spatial moments of the function $h_t$. Such relations are often referred to as \emph{level-inequalities}, since they establish relations between the Fourier mass in different energy levels. 

The main new point in this work is that, as it turns out, when we only look for a lower bound on (decoupled) moments, one can improve the bounds which appear in the literature, which give two sided but worse estimates.

For a random vector $X=(X_1,...,X_n)$ in $\RR^n$ which is absolutely continuous with respect to $\gamma$, with density $\rho d \gamma$, we define the relative entropy of $X$ with respect to $\gamma$ as
$$
\KL(X || \gamma) = \int \rho \log (\rho) d \gamma.
$$

At the heart of our proofs is the following lemma.
\begin{lemma} \label{lem:lvl21}
Let $X,Y$ be random vectors in $\RR^n$. Denote 
$$
H_X = \EE \left ( X^{\otimes 2}  - \Id \right ), ~~ H_Y = \EE \left ( Y^{\otimes 2}  - \Id \right ).
$$
Then,
$$
\Tr(H_X H_Y) \geq - 20 \Bigl ( \KL(X||\gamma) + \KL(Y||\gamma) \Bigr ).
$$
\end{lemma}

For the proof of this lemma, we need the entropy-transportation inequality due to Talagrand \cite{Talagrand96}. For two random vectors $X$ and $Y$ in $\RR^n$ we define the Wasserstein-2 distance between them as
$$
\mathrm{W}_2(X,Y) = \inf_{(\tilde X,\tilde Y)} \sqrt{\EE \left [ |\tilde X - \tilde Y|_2^2  \right ]},
$$
where the infimum is taken over all random vectors $(\tilde X, \tilde Y)$ in $\RR^{2n}$ whose marginals on the first and last $n$ coordinates are equal to $X$ and $Y$ respectively. The following is proven in \cite{Talagrand96}.

\begin{theorem} \label{thm:transport-ent}
Let $\Gamma$ be a standard Gaussian random vector in $\RR^n$ and let $X$ be a random vector such that $\KL(X || \gamma) < \infty$. Then,
$$
\mathrm{W}_2 (X, \Gamma)^2 \leq 2 \KL(X || \Gamma).
$$
\end{theorem}
\begin{remark}
For our proof, we will effectively only use the one-dimensional version of the above theorem.
\end{remark}

\begin{proof}[Proof of Lemma \ref{lem:lvl21}]
	By applying a rotation, we may assume without loss of generality that $H_X$ is diagonal. Let $\Gamma$ be a standard Gaussian random variable, and denote $\alpha_i = \mathrm{W}_2^2 (X_i, \Gamma)$ and $\beta_i = \mathrm{W}_2^2 (Y_i, \Gamma)$. Theorem \ref{thm:transport-ent} implies that
	$$
	\sum_{i \in [n]} \left ( \alpha_i + \beta_i \right ) \leq 2 \left ( \KL(X||\gamma) + \KL(Y||\gamma) \right ).
	$$
	Denote $\lambda_i = \EE[X_i^2]$ and $\delta_i = \EE[Y_i^2]$. Let $I = \{i \in [n]; ~ \lambda_i > 1 > \delta_i\}$ and $J = \{i \in [n]; ~ \delta_i > 1 > \lambda_i\}$. We clearly have,
	$$
	\Tr(H_X H_Y) = \sum_{i \in [n]} (\lambda_i - 1)(\delta_i - 1) \geq \sum_{i \in I \cup J} (\lambda_i - 1)(\delta_i - 1).
	$$
	Fix $i \in I \cup J$. The proof will be concluded by showing that 
	\begin{equation}\label{eq:lambdadelta}
	(\lambda_i - 1)(\delta_i - 1) \geq - 10 (\alpha_i + \beta_i).
	\end{equation}
	
	Assume that $i \in I$. The proof for the case $i \in J$ will be analogous. First suppose that $\lambda_i \geq 4$. In that case, 
	$$
	\lambda_i \leq 2(\lambda_i - 2) = 2 \left (\EE X_i^2 - 2 \right ) \leq 2\left (2 \mathrm{W}_2^2 (X_i, \Gamma) + 2 \mathrm{W}_2^2 (\Gamma, 0) - 2 \right ) = 4 \alpha_i.
	$$
	Since by assumption we have $\delta_i < 1$, we get
	\begin{align*}
	\left |(\lambda_i - 1)(\delta_i - 1) \right  | \leq \lambda_i \leq 4 \alpha_i,
	\end{align*}
	which establishes \eqref{eq:lambdadelta}. It remains to consider the case $\lambda_i \leq 4$. Since $\mathrm{W}_2$ is a metric, the triangle inequality implies that $X_i$ and $Y_i$ can be coupled in a way that
	\begin{equation}\label{eq:XYclose}
	\EE[ (X_i - Y_i)^2 ] \leq 2 (\alpha_i + \beta_i).
	\end{equation}
	We also have that
	\begin{equation}\label{eq:sumXY}
	\EE \left [ (X_i + Y_i)^2 \right ] \leq 2 (\lambda_i + \delta_i) \leq 10.
	\end{equation}
	By Cauchy-Schwartz, we have
	\begin{align*}
	2 (\lambda_i - 1)(\delta_i - 1) ~& = (\lambda_i - 1)^2 + (\delta_i - 1)^2 - (\lambda_i - \delta_i)^2 \\
	& \geq - \left (\EE (X_i^2 - Y_i^2) \right )^2 \\
	& = - \bigl (\EE \left [ (X_i - Y_i) (X_i + Y_i) \right ] \bigr )^2 \\
	& \geq - \EE \left [(X_i - Y_i)^2 \right ] \EE \left [ (X_i + Y_i)^2 \right ] \stackrel{\eqref{eq:XYclose} \wedge \eqref{eq:sumXY}}{\geq} -20 (\alpha_i + \beta_i).
	\end{align*}
	Equation \eqref{eq:lambdadelta} follows and the proof is complete.
\end{proof}

\begin{remark}
	Lemma \ref{lem:lvl21} only gives a \emph{lower} bound on the expression $\Tr(H_X H_Y)$. It is not hard to see that a matching upper bound on this expression will not hold true in general, in fact, the best upper bound attainable is
	$$
	\Tr(H_X H_Y) \leq C \KL(X || \gamma) \KL (Y || \gamma).
	$$
	The fact that the lower bound is better is crucial for the proof of our main theorem, and this lower bound lies in the heart of the reason that better bounds for correlations can be attained.
\end{remark}

The above lemma gives us an inequality between the entropy and the second-degree Hermite-Fourier coefficients, which in the case of indicators of sets, can be understood as an inequality between the zeroth and 2nd moments. Our next objective is to "lift" this inequality into an inequality between the first and third levels of energy, valid for monotone sets. This "lifting" is the essence of the main step in \cite{Talagrand96}. However, we are able to provide a shorter and simpler argument towards this lifting, also due to the fact that we work in the Gaussian setting.

For a measurable $f: \RR^n \to \RR$, define
$$
Q^{(k)}(f) = \int f(x) H^{(k)}(x) d \gamma(x).
$$
Our estimate reads,
\begin{proposition} (Level 1:3 inequality) \label{prop:level2}
Let $f,g :\RR^n \to \RR$ be two monotone functions. Then,
$$
\left \langle Q^{(3)}(f),  Q^{(3)}(g) \right \rangle \geq - e^8 \Bigl \langle Q^{(1)}(f), Q^{(1)}(g) \Bigl \rangle \log \left ( \frac{e \sqrt{ \Var[f] \Var[g] }}{ \left  \langle Q^{(1)}(f), Q^{(1)}(g) \right \rangle }\right ).
$$
\end{proposition}

The proof of Proposition \ref{prop:level2} is based on the following vectorial inequality.
\begin{lemma} \label{lem:geomineq}
	Let $v(x),u(x): \RR^n \to \RR_+^k$, such that
	$$
	\left \langle \int v(x) d \gamma(x), \int u(x) d \gamma(x) \right \rangle \leq \eps.
	$$
	Then one has
	\begin{equation}\label{eq:mom3}
	\Bigl \langle \int  v(x) \otimes (x^{\otimes 2} - \Id) d \gamma(x),  \int u(x) \otimes (x^{\otimes 2} - \Id) d \gamma(x) \Bigr \rangle \geq -  20 \eps \log \left  (\frac{ \int |v|_2^2 d \gamma \int |u|_2^2 d \gamma }{\eps^2} \right ).
	\end{equation}
\end{lemma}
\begin{proof}
	By the monotonicity of the right hand side of \eqref{eq:mom3} with respect to $\eps$, we may clearly assume that $\sum_i \int v_i(x) d \gamma(x) \int u_i(x) d \gamma(x) = \eps$. For $i \in [n]$, we denote
	$$
	\alpha_i := \frac{\Tr \left ( \int \left (x^{\otimes 2} - \Id \right ) v_i(x) d \gamma(x) \int \left (x^{\otimes 2} - \Id \right ) u_i(x) d \gamma(x) \right ) }{ \int v_i d \gamma \int u_i d \gamma } 
	$$
	(and $\alpha_i = 0$ when the denominator is zero). Remark that the left hand side of \eqref{eq:mom3} is equal to 
	$$
	\sum_i \Tr \left ( \int \left (x^{\otimes 2} - \Id \right ) v_i(x) d \gamma(x) \int \left (x^{\otimes 2} - \Id \right ) u_i(x) d \gamma(x) \right ),
	$$
	so the lemma will be concluded by showing that 
	\begin{equation}\label{eq:nts3}
	\sum_i \alpha_i \int v_i d \gamma \int u_i d \gamma \geq - 20 \eps \log \left  (\frac{ \int |v|_2^2 d \gamma \int |u|_2^2 d \gamma }{\eps^2}  \right ).
	\end{equation}	
	Define $\tilde v_i(x) = \frac{v_i(x)}{\int v_i d \gamma}$ and $\tilde u_i(x)$ similarly. An application of Lemma \ref{lem:lvl21} gives
	$$
	\alpha_i \geq - 20 \left (\KL(\tilde v_i, \gamma) + \KL(\tilde u_i, \gamma) \right ) \geq - 20 \left (\log \int \tilde v_i(x)^2 d \gamma + \log \int \tilde u_i(x)^2 d \gamma \right ),
	$$
	where the second inequality uses Jensen's inequality.
	This gives
	$$
	e^{-\frac{1}{40}\alpha_i} \int v_i(x) d \gamma \int u_i(x) d \gamma \leq \sqrt{\int v_i(x)^2 d \gamma \int u_i(x)^2 d \gamma}. 
	$$
	Denote $\beta_i = \frac{1}{\eps} \int v_i(x) d \gamma \int u_i(x) d \gamma$. The assumption $\sum_i \int v_i(x) d \gamma \int u_i(x) d \gamma  = \eps$ gives $\sum_i \beta_i = 1$. Now, that last display gives
	$$
	\sum_i e^{- \frac{1}{40} \alpha_i} \beta_i \leq 1/\eps \sum_i \sqrt{\int v_i(x)^2 d \gamma \int u_i(x)^2 d \gamma} \leq \frac{1}{\eps} \sqrt{\int |v|_2^2 d \gamma \int |u|_2^2 d \gamma } 
	$$
	and thus, by convexity,
	$$
	- \eps \sum_i \beta_i \alpha_i \leq 40 \eps \log \left ( \sum_i \beta_i e^{-\frac{1}{40} \alpha_i}   \right ) \leq 40 \eps \log \left ( \frac{1}{\eps} \sqrt{\int |v|_2^2 d \gamma \int |u|_2^2 d \gamma }  \right ),
	$$
	which implies \eqref{eq:nts3}, and finishes the proof.
\end{proof}

\begin{proof}[Proof of Proposition \ref{prop:level2}]
	Fix some $t>0$ and define
	$$
	f_t(x) = P_t[f](x), ~~ g_t(x) = P_t[g](x)
	$$
	where $P_t$ is the Ornstein-Uhlenbeck semigroup,
	\begin{equation}\label{eq:defOU}
	P_t[f](x) := \int f \left ( e^{-t/2} x + \sqrt{1-e^{t}}y \right ) d \gamma(y).
	\end{equation}
	Remark that $H^{(k)}(x)$ is an eigenfunction of the generator $L = \Delta - x \cdot \nabla$ with eigenvalue $-k$. Since $P_t$ is self adjoint, we have
	$$
	\int  H^{(3)}(x) P_t[f] (x) d \gamma(x) = \int f(x) P_t [H^{(3)}](x) d \gamma(x) = e^{-3t} Q^{(3)}(f).
	$$
	Further note that integration by parts yields,
	$$
	\int H^{(3)}(x) f_t(x) d \gamma(x) = \int H^{(2)}(x) \otimes \nabla f_t(x) d \gamma(x).
	$$
	The last two displays yield
	\begin{equation}\label{eq:Q3bp}
	Q^{(3)}(f) = e^{3t}  \int \left (x^{\otimes 2} - \Id \right ) \otimes \nabla f_t(x) d \gamma(x).
	\end{equation}
	A similar argument gives,
	\begin{equation}\label{eq:Q1bp}
	Q^{(1)}(f) = e^t \int x f_t(x) d \gamma(x) = e^t \int \nabla f_t(x) d \gamma(x).
	\end{equation}
	Moreover, consider the Hermite decomposition $f(x) = \sum_{\ell} \alpha_\ell H_\ell(x)$, so that $f_t = \sum_{\ell} \alpha_\ell e^{- |\ell| t} H_\ell(x)$, we have
	$$
	\int |\nabla f_t|^2 d \gamma = - \int f_t L f_t d \gamma = \sum_{\ell} e^{-2 |\ell| t} \alpha_\ell^2 |\ell| \leq \left (\sup_{q\geq 0} q e^{-2 t q} \right ) \sum_{\ell=1}^\infty  \alpha_\ell^2 \leq \frac{\Var[f]}{2te}.
	$$
	Taking $t=1$ and invoking Lemma \ref{lem:geomineq} with $u(x) = \nabla f_t(x)$ and $v(x) = \nabla g_t(x)$ gives
	\begin{align*}
	\left \langle Q^{(3)}(f), Q^{(3)}(g) \right \rangle  &~~ \stackrel{ \eqref{eq:Q3bp} }{=} e^6 \Tr \left ( \left (\int \left (x^{\otimes 2} - \Id \right ) \otimes \nabla f_1(x) d \gamma(x) \right ) \left  ( \int \left (x^{\otimes 2} - \Id \right ) \otimes \nabla g_1(x) d \gamma(x).   \right ) \right ) \\
	& \stackrel{ \eqref{eq:mom3} \wedge \eqref{eq:Q1bp} }{\geq} - 40 e^{6} \left (e^{-2} \langle Q^{(1)}(f), Q^{(1)}(g) \rangle \log \left ( \frac{e \sqrt{\Var[f] \Var[g]}  }{2 \langle Q^{(1)}(f), Q^{(1)}(g) \rangle } \right )  \right).
	\end{align*}
	This completes the proof.
\end{proof}

\begin{remark}
A very small modification of the above proof, where Lemma \ref{lem:lvl21} is replaced by the level-1 inequality reproduces the Gaussian variant of the bound between levels 1 and 2 due to Talagrand, via an arguably simpler route. 
\end{remark}

\section{Proof of the main bounds}

In this section we proof our main bounds in the Gaussian setting. We begin with a technical lemma whose proof is postpones to the end of the section.
\begin{lemma} \label{lem:gronwall}
	Let $p:[0,\infty) \to [0,1]$ be twice differentiable. Suppose that $p(0) \in (0,1)$ and that there exists $K>0$ such that for all $x$,
	$$
	p''(t) \geq -p'(t) - K p(t) \log (e/p(t)).
	$$
	Then for all $t \leq \min \left (\frac{1}{4 \sqrt{K\log(2e/p(0))}}, \frac{p(0)}{4 |p'(0)|} \right )$ one has $p(t) \geq p(0) / 2$.
\end{lemma}
We are now ready to prove our theorems.

\begin{proof}[Proof of Theorem \ref{thm:mainboundTalGauss}]
Define 
$$
p(t) := \EE \left [ \left \langle  M_t^{(f,1)}, M_t^{(g,1)} \right \rangle \right ],
$$
where $M_t^{(f,1)}, M_t^{(g,1)}$ are defined as in equation \eqref{eq:defmt1}.
Equation \eqref{eq:ptdd2} gives
$$
p''(t) = - p'(t) + e^{-2t} \EE \left [\left \langle M_t^{(f,3)}, M_t^{(g,3)} \right \rangle \right ].
$$
Remark that $M_t^{(f, k)} = e^{kt/2} Q^{(k)} (f_t)$, where $f_t$ is defined as in \eqref{eq:defht}, and the same is true for $g$. An application of Proposition \ref{prop:level2} on the functions $f_t$ and $g_t$ gives that almost surely, for all $t$, one has
$$
\left \langle M_t^{(f,3)}, M_t^{(g,3)} \right \rangle \geq - e^{8+2t} \left  \langle  M_t^{(f,1)}, M_t^{(g,1)} \right \rangle \log \left ( \frac{e^t}{ \left \langle  M_t^{(f,1)}, M_t^{(g,1)} \right \rangle } \right ).
$$
Since $s \to - s \log (1/s)$ is convex, Jensen's inequality gives
$$
\EE \left [\left \langle M_t^{(f,3)}, M_t^{(g,3)} \right \rangle \right ] \geq - e^{8+2t} p(t) \log (e^t/p(t)),
$$
concluding that 
$$
p''(t) \geq -p'(t) - e^8 p(t) \log (e/p(t)), ~~ \forall t \leq 1.
$$ 
Equations \eqref{eq:MtMh} and \eqref{eq:ptdd1} give
$$
p(0) = \langle \WW_1(f), \WW_1(g) \rangle, ~~  p'(0) = \left \langle \WW_2(f), \WW_2(g) \right \rangle_{HS}.
$$
Via an application of lemma \ref{lem:gronwall}, the two last displays imply that 
$$
p(t) \geq \frac{1}{2} \langle \WW_1(f), \WW_1(g) \rangle, ~~ \forall t \leq \min \left (\frac{e^{-8}}{\sqrt{ \log\left (\frac{2e}{\langle \WW_1(f), \WW_1(g) \rangle} \right ) }}, \frac{\langle \WW_1(f), \WW_1(g) \rangle}{4 |\left \langle \WW_2(f), \WW_2(g) \right \rangle_{HS}|} \right ).
$$ 
Plugging this estimate into Proposition \ref{prop:intcov} finishes the proof.
\end{proof}

\begin{proof}[Proof of Theorem \ref{thm:mainboundG}]
First assume that $f,g$ are $C_\infty$-differentiable. Fix $i \in [n]$. Integration by parts gives that
$$
M_t^{\partial_i f} = e^{t/2} \int \partial_i f_t d \gamma = e^{t/2} \int x_i f_t d \gamma = \langle M_t^{(f,1)}, e_i \rangle.
$$
Proposition \eqref{prop:intcov} therefore gives that for all $t_0 > 0$,
\begin{equation}\label{eq:intcov2}
\Cor(f,g) \geq  \sum_i \int_0^{t_0} e^{-t} \EE[M_t^{\partial_i f} M_t^{\partial_i g}] dt. 
\end{equation}
Now,
$$
M_t^{(\partial_i f,2)} = e^{t} \int (\partial_i f)(L_t(x)) (x^{\otimes 2} - \Id) d \gamma = e^{3t/2} \int \partial_i f_t (x^{\otimes 2} - \Id) d \gamma.
$$
Denote $\ff(x) = P_1[\partial_i f_t]$, where $P_t$ is the Ornstein-Uhlenbeck operator defined as in \eqref{eq:defOU}. Since $H^{(k)}(x)$ are eigenfunctions of $P_t$ with eigenvalues $e^{-kt}$, we have
\begin{equation}\label{eq:Ptev1}
M_t^{(\partial_i f)} = \int P_1[\partial_i f_t] d \gamma
\end{equation}
and
\begin{equation}\label{eq:Ptev2}
M_t^{(\partial_i f,2)} = e^{2} e^{3t/2} \int P_1[\partial_i f_t](x) (x^{\otimes 2} - \Id) d \gamma.
\end{equation}
Since $f_t(x) \in [0,1]$ for all $x$, it is easy to verify that $P_1 [\partial_i f_t] \in [0,2]$ for all $x$ and thus
$$
\int \frac{\ff(x)}{\int \ff d \gamma} \log \frac{\ff(x)}{\int \ff d \gamma} d \gamma(x) \leq \log \frac{\int \ff(x)^2 d \gamma(x)}{\left (\int \ff(x) d \gamma \right )^2} \leq \log \left ( \frac{2}{\int \ff d \gamma} \right ) = \log \left ( \frac{2}{M_t^{\partial_i f}} \right ).
$$
An application of Lemma \ref{lem:lvl21} for $X,Y$ being distributed according to the laws $\frac{\ff(x)}{\int \ff d \gamma}, \frac{\gg(x)}{\int \gg d \gamma}$ respectively thus gives
$$
\left \langle Q^{(2)}(\ff), Q^{(2)}(\gg) \right \rangle \geq -20 M_t^{\partial_i f} M_t^{\partial_i g} \log \left ( \frac{1}{M_t^{\partial_i f} M_t^{\partial_i g}} \right ).
$$
Together with equation \eqref{eq:Ptev2}, this gives
\begin{equation}\label{eq:prejensen}
\left \langle M_t^{(\partial_i f,2)}, M_t^{(\partial_i f,2)} \right \rangle_{HS} \geq -20 e^{2+3t/2} M_t^{\partial_i f} M_t^{\partial_i g} \log \left ( \frac{1}{M_t^{\partial_i f} M_t^{\partial_i g}} \right ).
\end{equation}
Define $S_t = M_t^{\partial_i f} M_t^{\partial_i g}$ and $p(t) = \EE[S_t]$. Equation \eqref{eq:Ptev1} and integration by parts gives
\begin{equation}\label{eq:ft2-1}
p(0) = \WW_1(f)_i \WW_1(g)_i.
\end{equation}
By equation \eqref{eq:ptdd2}, we have
$$
p''(t) = -p'(t) + e^{-2t} \EE  \left [\left \langle M_t^{(\partial_i f,2)}, M_t^{(\partial_i f,2)} \right \rangle_{HS} \right ].
$$
The bound \eqref{eq:prejensen} and the concavity of the function $x \to x \log(1/x)$ yield
\begin{equation}\label{eq:ft2-2}
p''(t) \geq - p(t) - 20 e^2 p(t).
\end{equation}
Now, remark that
$$
M_0^{(\partial_i f,1)} = \int x (\partial_i f) d \gamma = \WW_2(f) e_i,
$$
which, along with equation \eqref{eq:ptdd1} implies that
$$
p'(0) = \left \langle M_0^{(\partial_i f,1)}, M_0^{(\gg,1)} \right \rangle = \langle \WW_2(f) e_i, \WW_2(g) e_i \rangle.
$$
Combining this with \eqref{eq:ft2-1} and \eqref{eq:ft2-2} and applying Lemma \ref{lem:gronwall} finally gives that
$$
p(t) \geq \frac{1}{2} \WW_1(f)_i \WW_1(g)_i, ~~ \forall t \leq \min \left (\frac{1}{20 e \sqrt{\log(2e/\WW_1(f)_i \WW_1(g)_i)}}, \frac{\WW_1(f)_i \WW_1(g)_i}{4 |\langle \WW_2(f) e_i, \WW_2(g) e_i \rangle|} \right ).
$$
Plugging this into equation \eqref{eq:intcov2} completes the proof in the case that $f,g$ are $C_\infty$-smooth. The general case easily follows by considering the functions $P_\delta[f],P_\delta[g]$ in place of $f,g$ and taking $\delta \to 0$.
\end{proof}

\begin{proof}[Proof of lemma \ref{lem:gronwall}]
Let $t_0 = \min\{t>0; ~p(t) \leq p(0)/2 \}$. We first claim that without loss of generality we may assume that 
\begin{equation}\label{eq:asuumpp}
p(0) = \max_{0 \leq s \leq t_0} p(s).
\end{equation}
Indeed, if we suppose that $p(t)$ has a local maximum at $m \in (0, t_0)$ where $p(m) \geq p(0)$, then we may define $\tilde p(t) = p(t+m)$ and proceed by replacing the function $p$ by $\tilde p$, using the fact that $\tilde p(0)' = 0$ and the monotonicity of the expression $x \to \frac{1}{\sqrt{\log(2e/x) }}$ on $[0,1]$.

Now, assume by contradiction that
\begin{equation}\label{eq:contrasump}
t_0 < \min \left (\frac{1}{2 \sqrt{K \log(2e/p(0))}}, \frac{p(0)}{4 |p'(0)|} \right ).
\end{equation}
By Largange's theorem, there exists $t_1 \in (0,t_0)$ for which $p'(t_1) \leq -\frac{p(0)}{2 t_0}$. If there exists $0 \leq t \leq t_1$ such that $p'(t) \geq 0$, set $s = \max\{t; p'(t) \geq 0 \}$, otherwise set $s=0$. Note that $p'(s) \geq -|p'(0)|$. Remark that by definition of $t_0$ and $t_1$, we have
$$
p'(t_1) \leq -\frac{p(0)}{2 t_0} \leq -2 |p'(0)| \leq 2 p'(s),
$$
which implies that
\begin{equation}\label{eq:pt1}
p'(t_1) - p'(s) \leq \frac{1}{2} p'(t_1).
\end{equation}
Applying Largange's theorem again, we have that there exists $t_2 \in (s,t_1)$ such that
$$
p''(t_2) \leq \frac{p'(t_1) - p'(s)}{t_1-s} \stackrel{ \eqref{eq:pt1} }{\leq} \frac{p'(t_1)}{2 t_0} \leq -\frac{p(0)}{4 t_0^2} \leq - K p(0) \log(2e/p(0)).
$$
However, by the fact that $t_2 \leq t_0$, and by \eqref{eq:asuumpp}, we have $p(t_2) \in [p(0)/2, p(0)]$, and therefore since $s \to s \log (e/s)$ is increasing on $[0,1]$,
$$
-K p(0) \log(2e/p(0)) \leq -K p(t_2) \log(e/p(t_2)).
$$
Recalling that $p'(t_2) \leq 0$, the two last displays contradict the assumption of the lemma. This completes the proof.
\end{proof}

\bibliographystyle{alpha}
\bibliography{bib}

\begin{thebibliography}{FKKK18}

\bibitem[Chv74]{Chvatal}
V.~Chv\'{a}tal.
\newblock Intersecting families of edges in hypergraphs having the hereditary
  property.
\newblock In {\em Hypergraph {S}eminar ({P}roc. {F}irst {W}orking {S}em.,
  {O}hio {S}tate {U}niv., {C}olumbus, {O}hio, 1972; dedicated to {A}rnold
  {R}oss)}, pages 61--66. Lecture Notes in Math., Vol. 411, 1974.

\bibitem[EG19]{EldanGrossBool}
Ronen Eldan and Renan Gross.
\newblock Concentration on the boolean hypercube via pathwise stochastic
  analysis, 2019.

\bibitem[Eld15]{Eldan-twosided}
Ronen Eldan.
\newblock A two-sided estimate for the {G}aussian noise stability deficit.
\newblock {\em Invent. Math.}, 201(2):561--624, 2015.

\bibitem[FKKK18]{FKKK}
Ehud Friedgut, Jeff Kahn, Gil Kalai, and Nathan Keller.
\newblock Chv\'{a}tal's conjecture and correlation inequalities.
\newblock {\em J. Combin. Theory Ser. A}, 156:22--43, 2018.

\bibitem[KKM16]{KKM15}
Gil Kalai, Nathan Keller, and Elchanan Mossel.
\newblock On the correlation of increasing families.
\newblock {\em J. Combin. Theory Ser. A}, 144:250--276, 2016.

\bibitem[KMS14]{KMS14}
Nathan Keller, Elchanan Mossel, and Arnab Sen.
\newblock Geometric influences {II}: correlation inequalities and noise
  sensitivity.
\newblock {\em Ann. Inst. Henri Poincar\'{e} Probab. Stat.}, 50(4):1121--1139,
  2014.

\bibitem[Tal96a]{Talagrand96}
M.~Talagrand.
\newblock Transportation cost for {G}aussian and other product measures.
\newblock {\em Geom. Funct. Anal.}, 6(3):587--600, 1996.

\bibitem[Tal96b]{Talcorr}
Michel Talagrand.
\newblock How much are increasing sets positively correlated?
\newblock {\em Combinatorica}, 16(2):243--258, 1996.

\end{thebibliography}

\end{document}